\pdfoutput=1

\documentclass[12pt]{article}
\usepackage[margin=1in]{geometry}
\usepackage{amsmath}
\usepackage{amsfonts}
\usepackage{amssymb}
\usepackage{amsthm}
\usepackage{esint}
\usepackage{mathtools}
\usepackage{array,makecell,longtable}
\usepackage{titling}
\usepackage{authblk}
\usepackage{commath}
\usepackage{xcolor}
\usepackage{graphicx}
\usepackage{graphbox}
\usepackage[toc,title,page]{appendix}
\usepackage{subcaption}
\usepackage{bbm}
\usepackage{hyperref}
\usepackage{color,soul}
\usepackage{bm}
\usepackage{mathrsfs}

\usepackage{enumitem}
\newtheorem{remark}{Remark}[section]
\newtheorem{theorem}[remark]{Theorem}

\newtheorem{lemma}[remark]{Lemma}

\newtheorem{definition}[remark]{Definition}

\numberwithin{equation}{section}

\title{Strong uniqueness principle for fractional polyharmonic operators and applications to inverse problems}
\author{
Ching-Lung Lin \thanks{Department of Mathematics, National Cheng Kung University, Taiwan\\ Email address: cllin2@mail.ncku.edu.tw}, \,
Hongyu Liu\thanks{Department of Mathematics, City University of Hong Kong, Hong Kong SAR, China\\ Email address: hongyu.liuip@gmail.com, hongyliu@cityu.edu.hk} \, 
and Catharine W. K. Lo\thanks{Liu Bie Ju Centre for Mathematical Sciences, City University of Hong Kong, Hong Kong SAR, China\\ Email address: wingkclo@cityu.edu.hk} \vspace{-0.5cm}
}
\date{}

\begin{document}

\maketitle

	\begin{abstract}
		In this work, we are concerned with inverse problems involving poly-fractional operators, where the poly-fractional operator is of the form 
        \[P( (-\Delta_g)^s)u := \sum_{i=1}^M \alpha_i(-\Delta_{g_i})^{s_i}u\] 
        for  $s=(s_1,\dots,s_M)$, $0<s_1<\cdots<s_M<\infty$, $s_M\in\mathbb{R}_+\backslash\mathbb{Z}$, $g=(g_1,\dots,g_M)$. There are three major contributions in this work that are new to the literature. First, we propose equations involving such poly-fractional operators $P$, which have not been previously considered in the general setting. Such equations arise naturally from the superposition of multiple stochastic processes with different scales, including classical random walks and L\'evy flights. Secondly, we give novel results for the unique continuation properties for fractional polyharmonic $u$, in the sense that $u$ satisfies $\tilde{P}((-\Delta_{\tilde{g}})^{\tilde{s}})=0$ in a bounded Lipschitz domain $\Omega$ for some $\tilde{P}$. With these results in hand, we consider the inverse problems for $P$, and proved the uniqueness in recovering the potential, the source function in the semilinear case, and the coefficients associated to the non-isotropy of the fractional operator.
		
		\medskip
		
		

	\end{abstract}
	
\section{Introduction}

\subsection{Mathematical Setup and Statement of the Main Results}

To provide a general picture of our study, we consider the following poly-fractional exterior value problem: 
\begin{equation}\label{GeneralProb}
    \begin{cases}
        P( (-\Delta_g)^s)u(x) + q(x)u(x) := \sum_{i=1}^M \alpha_i(x)(-\Delta_{g_i})^{s_i}u(x) + q(x)u(x) =0 &\text{ in }\Omega, \\
        u(x)=f(x) &\text{ in }\Omega^c:=\mathbb{R}^n\backslash \overline{\Omega},
    \end{cases}
\end{equation}
where $\Omega\subset\mathbb{R}^n$, $n\in \mathbb{N}$, is a bounded Lipschitz domain, $q\in L^\infty(\Omega)$, $s=(s_1,\dots,s_M)$, $0< s_1<\cdots<s_M<\infty$, $s_M\in\mathbb{R}_+\backslash\mathbb{Z}$. Note that the other exponents $s_m$ of the Laplacian, $m=1,\dots,M-1$, may be fractional or integer powers. Here, $(-\Delta_\gamma)^\sigma$ represents the anisotropic fractional Laplacian for all $0<\sigma<\infty$, which is defined in detail later in Section \ref{subsec:OperatorDef}. Classically, the fractional Laplacian is defined for $0<\sigma<1$, and the higher order fractional Laplacian was first investigated in \cite{ChangGonzalez2011AdvMath-HigherOrderFracLap} and \cite{GonzalezSaez2018HigherOrderFracLap} using conformal geometry techniques, and later developed in \cite{Yang2013HigherOrderFracLap} and \cite{CoraMusina2022JFA-HigherOrderFracLap}. 

These equations are a natural result of combining various stochastic processes with different scales, encompassing both classical random walks and L\'evy flights. Furthermore, \eqref{GeneralProb} can also be viewed as a combination of both nonlocal operators and local operators, in the form 
\begin{equation*}
    \begin{cases}
        \hat{P}( (-\Delta_g)^s)u(x) + \hat{L}( (-\Delta_g))u(x) + q(x)u(x) =0 &\text{ in }\Omega, \\
        u(x)=f(x) &\text{ in }\Omega^c:=\mathbb{R}^n\backslash \overline{\Omega},
    \end{cases}
\end{equation*}
where $\hat{P}$ is a purely nonlocal poly-fractional operator, while $\hat{L}$ is a purely local operator, such that the order of $\hat{L}$ is less than that of $\hat{P}$. Consequently, this covers the case of a single higher order fractional Laplacian with lower order local perturbation in \cite{CMRU2022HigherOrderFracPolyUniqueness}. More discussion shall be given in the next subsection about this aspect. 

As such, in this paper, we are mainly concerned with the inverse problem of recovering $P$ and $q$, using knowledge of the exterior value of the solution $u$ and its poly-fractional harmonicity. Physically, this corresponds to recovering various diffusion properties in the multiple stochastic processes. To this end, we introduce the Dirichlet-to-Neumann (DtN) map formally via 
\begin{equation}\label{eq:ip1}
    \mathcal{M}_{P,q} : H^{s_M}(\Omega^c)\to H^{-s_M}(\Omega), \quad f\mapsto \left. \tilde{P}( (-\Delta_{\tilde{g}})^{\tilde{s}}) u_f \right|_{\Omega} ,\quad j=1,2,
\end{equation}
where $u_f \in H^s(\mathbb{R}^n)$ is the unique solution to \eqref{GeneralProb}, and $\tilde{P}\in\mathcal{A}$ is another poly-fractional operator, which may be the same or different from $P$, for some admissible class $\mathcal{A}$ which will be detailed later.

For the inverse problem \eqref{eq:ip1}, we are mainly concerned with the theoretical unique identifiability issue, which is of primary importance for a generic inverse problem. In its general formulation, the unique identifiability asks whether one can establish the following one-to-one correspondence for two configurations $(P^j,q_j)$, $j = 1, 2$:
\begin{equation}\label{MainUniquessnessProb}
\mathcal{M}_{P^1,q_1} = \mathcal{M}_{P^2,q_2} \quad\text{ if and only if }\quad (P^1,q_1) = (P^2,q_2).
\end{equation} 
 In this paper, we aim to prove, in formal terms, the following theorem. 

\begin{theorem}\label{FormalThm}
    Let $\mathcal{M}_{P^j,q_j}$ be the measurement map associated to \eqref{GeneralProb} for $j=1,2$. Suppose that for a given nonzero $f$, where $f$ is properly chosen to satisfy some properties,
    \[
		\mathcal{M}_{P^1,q_1} f = \mathcal{M}_{P^2,q_2} f.
    \] 
    Then 
    \[(P^1, q_1) = (P^2, q_2)\] in some subset of $\Omega$.
\end{theorem}

We assume that $P( (-\Delta_g)^s) + q$ is a coercive bounded operator with domain $H^{s_M}(\mathbb{R}^n)$, such that there exists a solution $u$ to \eqref{main}. Some possibilities are detailed in Section \ref{subsec:OperatorDef}. We also assume that $0$ is not an eigenvalue of the operator $(P( (-\Delta_g)^s) + q)$, i.e.
\begin{equation}\label{qCond}
    \begin{cases}
        \text{ if }w\in H^{s_M}(\mathbb{R}^n) \text{ solves }(P( (-\Delta_g)^s) + q)w=0 \text{ in }\Omega \text{ and }w|_{\Omega^c}=0,\\
        \text{ then }w\equiv0.
    \end{cases}
\end{equation}

To prove Theorem \ref{FormalThm}, an important ingredient is the unique continuation principle (UCP). This is formally given as: 
\begin{theorem}\label{UCPFormal}
For sufficiently regular $u$ satisfying some additional assumptions, the condition
\begin{equation}\label{UCPCondFormal}
		\begin{cases}
            \tilde{P}( (-\Delta_{\tilde{g}})^{\tilde{s}})u = 0 & \text{ in }\Omega,\\
            u=0 & \text{ in } W,
		\end{cases}
    \end{equation} 
for some properly chosen $W\subseteq \Omega^c$ implies  $u\equiv 0$ in $\mathbb{R}^n$. 
\end{theorem}
The concrete form of this UCP result is given in Theorems \ref{UCPThm2}--\ref{UCPThm3}. Note that this is different from the usual UCP results for fractional equations, which are commonly based on the form used by \cite{GSU2020CalderonSchrodinger}, where $\tilde{P}u=u=0$ in $W$ for any open set $W$.

Such a UCP result may have physical significance, such as in applications to astronomy and imaging. Consider the inverse problem of determining the processes (nonlocal and local) occurring in a celestial body. Then, according to \eqref{eq:ip1}, we measure some other nonlocal process, that we can observe, in the celestial body, as well as the radiative spectrum it creates in the exterior background. For instance, suppose we are able to obtain some information for some nonlocal process that occurred in the opaque fog of dense, hot plasma of sub-atomic particles of the primordial universe. Together with the cosmic microwave background radiation heat map, we are able to determine the properties of the nonlocal and local processes that happened in the fog. See, for instance, \cite{CMB1}, \cite{CMB2}, \cite{CMB3} for more explanation of nonlocal processes occurring in astronomy. Suppose that the gas particles or radiation distribute in the exterior via a diffusion-type process, such that it obeys a second order elliptic equation. Then, we can simplify the exterior measurement, by measuring only in a small domain $W$. An example of such a $W$ is any region in space, such as in Earth. Such a measurement, together with the measurement of any nonlocal process in the celestial body, will enable us to recover all the local and nonlocal processes happening in that celestial body.

Another practical scenario is thermo-acoustic and photo-acoustic tomography. See, for instance, \cite{StefanovUhlmann2009ThermoacousticTom} and \cite{LiuUhlmann2015DetermineSpeedHyperbolic}. Here, a short electro-magnetic pulse is sent through a patient’s body. The tissue reacts, and an ultrasound wave is emitted from any point. This wave undergoes another series of absorption while leaving the body, which can be represented by $\tilde{P}u$. We also measure the wave outside the body, which can be assumed to be acoustically homogeneous. Then one tries to reconstruct the internal structure $Pu$ of a patient’s body from those measurements. By including both local and nonlocal operators in $P$ and $\tilde{P}$, we allow the absorption and emission to involve lossy media which exhibit fractional damping effects (as detailed in \cite{ChenHolm2004FracLossyMedia}, \cite{CaoLiu2019FracHelmholtz} and \cite{KaltenbacherRundell2021FracWaveInverse}). 

We shall establish sufficient conditions such that the Theorems \ref{FormalThm} and \ref{UCPFormal} hold in a certain general setup. In particular, we shall provide general characterisations of $P, \tilde{P}$ and $\mathcal{A}$. The major novelty that distinguishes our inverse problem study from most of the existing ones lies in the following three aspects. First, we study a novel class of nonlocal equations defined with poly-fractional operators. Second, we innovate a new type of unique continuation property given by \eqref{UCPCondFormal}. With these in hand, thirdly, we can consider the exterior value inverse problems associated to poly-fractional operators.

\subsection{Discussions and Historical Remarks}

The study of inverse problems in the context of partial differential equations (PDEs) has long fascinated researchers. The recovery of internal properties of a medium (corresponding to certain terms in an equation/system or operator) from indirect measurements (corresponding to information on solutions of equations in certain domains) remains a pertinent problem in many scientific disciplines such as electromagnetism, geophysics, medical imaging and economics. Consequently, the study of inverse problems associated with partial differential equations remains an active and influential research area.

One of the most famous problems in this area is the Calder\'on problem arising in electrostatics. The classical Calder\'on problem investigates whether one can determine the electrical conductivity $\gamma(x)$ of a medium by making voltage and current measurements at its boundary. It is modeled by the following Dirichlet problem:
\[
\begin{cases}
    \nabla\cdot (\gamma\nabla u) = 0& \text{ in }\Omega,\\
    u = f & \text{ on }\partial\Omega,
\end{cases}\]
where the conductor filling $\Omega$ is a bounded domain with smooth boundary. In mathematical terms, the Calder\'on problem asks whether one can determine $\gamma$ from the knowledge of the Dirichlet-to-Neumann map defined by 
\[\Lambda_\gamma: f\mapsto \left. \gamma \frac{\partial u}{\partial \nu}\right|_{\partial\Omega}.\] 
Physically, this means that we apply a voltage $f$ at the boundary $\partial\Omega$, which will induce a voltage $u(x)$ in $\Omega$, and we measure the current $\gamma \frac{\partial u}{\partial \nu}$  at the boundary $\partial\Omega$.

Beginning with the seminal work of Calder\'on in \cite{Calderon1980}, the inverse conductivity problem has been studied intensively. Numerous positive result have been obtained, including in \cite{KohnVogelius,KV85}, \cite{sylvester1987global}, \cite{Alessandrini1988} and \cite{kenig2007calderon}. In particular, in \cite{Alessandrini1988}, Alessandrini reduced the conductivity-type problem to a Schr\"odinger-type one, where one attempts to determine the potential $q(x)$ in 
\[
\begin{cases}
    - \Delta v + qv = 0& \text{ in }\Omega,\\
    v = f & \text{ on }\partial\Omega
\end{cases}\]
from the measurement map 
\[\Lambda_q: f\mapsto \left. \frac{\partial v}{\partial \nu}\right|_{\partial\Omega}.\] 

Recently, the study of equations involving non-local operators has gained substantial attention. A typical non-local operator is the fractional Laplacian $(-\Delta)^s$. These kinds of equations are interesting due to their capability to model complex systems. Such effects arise in a diverse range of disciplines, in which the presence of anomalous diffusion effects, long-range correlations, and memory effects necessitates the consideration of fractional problems. By applying the concept of fractional calculus to problems in control theory, optimization, image processing, structural dynamics, signal processing, epidemiology, and population dynamics, one is able to model and analyse better complex physical phenomena. 

Correspondingly, inverse problems associated with fractional operators have been studied. A major point of interest is the fractional Schr\"odinger equation in the field of fractional quantum mechanics, which arises naturally as a generalisation of the classical Schr\"odinger equation. The Calder\'on problem for the fractional Schr\"odinger equation was first solved by Ghosh, Salo and Uhlmann in \cite{GSU2020CalderonSchrodinger}. In this work, instead of the (boundary value) Dirichlet problem associated with the classical Calder\'on problem, the authors considered the exterior value Dirichlet problem
\[
\begin{cases}
    (- \Delta)^\sigma u + qu = 0& \text{ in }\Omega,\\
    u = f & \text{ in }\Omega^c,
\end{cases}\] 
for $\sigma\in(0,2)$. The inverse problem asks whether one can determine the potential $q$ in $\Omega$ from the exterior partial measurements of the Dirichlet-to-Neumann map 
\[\Lambda_q: f\mapsto \left. (- \Delta)^\sigma u\right|_{\Omega^c}.\] 

This problem has a positive answer in \cite{GSU2020CalderonSchrodinger}, where the Dirichlet-to-Neumann map $\Lambda_q$ uniquely determines $q$ in $\Omega$. This result was then generalised in numerous works in many different directions, including in \cite{lai2019global}, \cite{Covi2020IP}, \cite{GRSU20},  \cite{bhattacharyya2021inverse}, \cite{CMR2021HigherOrderFracLapUCP}, \cite{Li2021CommPDEFracMagneticPotential}, \cite{Ruland2021SingleMeasurementStability}, \cite{CMRU2022HigherOrderFracPolyUniqueness} and \cite{CGRU2023reduction}, to name a few. The proof of the fractional Calder\'on problem strongly relies on the strong uniqueness property: for $u\in H^{\sigma/2}(\mathbb{R}^n)$, 
\[u = \mathcal{L}^\sigma u = 0 \text{ in an arbitrary nonempty
open set in }\mathbb{R}^n \quad\text{ implies }\quad u\equiv0 \text{ in }\mathbb{R}^n\] for appropriately defined fractional Laplacian-type operators $\mathcal{L}^\sigma$.  
The proof of the unique continuation property above is based on the Caffarelli-Silvestre definition \cite{CaffarelliSilvestre2007CommPDE-Extension} of the fractional Laplacian 
\[(-\Delta)^\sigma u(x) := C_{n,\sigma}\lim_{y\to 0^+} y^{1-2\sigma}\frac{\partial}{\partial y}  U(x,y) \text{ for }x\in \mathbb{R}^n\quad\text{ for some constant }C_{n,\sigma},\]
where $U$ is the solution of the extension problem
\[
	\begin{cases}
		\nabla \cdot (y^{1-2\sigma}\nabla U )=0 & \text{ in }\mathbb{R}^{n+1}_+,\\
		U (x,0)=u(x) & \text{ in }\mathbb{R}^n,
	\end{cases}
\]
This definition enables us to derive properties of the fractional Laplacian $(-\Delta)^\sigma$ from local arguments in the extension problem, such as in \cite{RulandCommPDE2015UCPFracSchrodingerRoughPotential}. 

We remark that there are distinct differences between the classical and fractional Calder\'on problems. In particular, no construction of CGO solutions is required in dealing with the fractional problem. On the other hand, the unique continuation property is a distinctive feature of fractional operators which is not present in local operators, and it makes fractional inverse problems more manageable and helps us obtain strong results. Indeed, it has been observed in \cite{CGRU2023reduction} that a uniqueness result in the local case guarantees a uniqueness result in the fractional setting, but the result does not hold vice versa.

Consequently, in the consideration of higher order fractional Laplacians $(-\Delta)^\sigma$ for $\sigma\in(0,\infty)$, many works have focused on the derivation of a unique continuation principle, such as in \cite{Yang2013HigherOrderFracLap}, \cite{SeoHigherOrderUCP1,SeoHigherOrderUCP2,SeoHigherOrderUCP3} , \cite{GarciaRuland2019UCPHigherOrderFracLap}, \cite{FelliFerrero2020HigherOrderFracLapUCP}, \cite{CMRU2022HigherOrderFracPolyUniqueness}, \cite{KarRailoZimmermann2023HigherOrderFracPLapUCP} and \cite{KowWangEigenUCPfrac}. Once again, many of these works relied on a Caffarelli-Silvestre-type extension, which has been extended to higher fractional exponents $\sigma$.

The higher order fractional Laplacian extends the fractional Laplacian by considering fractional orders $\sigma$ greater than 2. Higher order fractional Laplacians provide a way to capture even more intricate details of the function's curvature and variations, and were first considered in geometrical settings in \cite{ChangGonzalez2011AdvMath-HigherOrderFracLap}. It is useful for analysing and modeling complex data with non-local dependencies, enabling a more accurate representation and understanding of intricate structures and patterns. Therefore, our consideration of inverse problems associated with the higher order fractional Laplacian $(-\Delta)^\sigma$ for $\sigma\in(0,\infty)$ is very pertinent and useful for many physical applications.

Another distinctive aspect of our work is the involvement of fractional operators with mixed orders, which we will call \emph{poly-fractional operators}. Most previous works have focused on fractional-type operators with a single order of singularity, i.e. the kernel of $\mathcal{L}^\sigma$ is of the form $K(x,y)|y|^{-d-\sigma}$ where $K(x,y)$ is homogeneous of order zero and sufficiently smooth in $y$. Fractional partial differential equations with mixed singularities are much less understood. Recently, various existence and regularity results for the forward problems involving mixed fractional operators have been obtained. This includes cases of sums of fractional Laplacians as in \cite{CabreSerra2016ExtensionProblem}, mixing a local operator with a nonlocal operator such as in \cite{BarlowBassChenKassmann2009MixedLocalNonlocalParabolic}, \cite{ChenKimSongVondracek2012MixedLocalNonlocalHarnack}, \cite{rosoton2015nonexistence}, \cite{garain2021regularity}, \cite{BiagiDipierroValdinociVecchi2022CommPDEMixedLocalNonlocal}, \cite{deFilippisMingione2022MixedLocalNonlocalPLap}, \cite{SVWZ2022MixedLocalNonlocal}, \cite{Biroud2023MixedLocalNonlocal}, \cite{byun2023regularity}, \cite{GarainKinnunen2023JDEMixedLocalNonlocalHarnack}, \cite{GarainLindgren2023CVPDEMixedLocalNonlocalHolder} or \cite{BiagiMeglioliPunzo2023-MixedLocalNonlocalUniqueness}, as well as cases where different orders of fractional operators are applied on the space and time variables separately such as in \cite{KimParkRyu2021JDEMixedSpaceTimeFrac}, \cite{kang2022lqlptheory} or  \cite{DongLiu2023CVPDEMixedSpaceTimeFrac}.

Operators with mixed fractional orders (including the classical local Laplacian) arise naturally from the superposition of multiple stochastic processes with different scales, including classical random walks and L\'evy flights. When a particle follows either of these processes according to a certain probability, the associated limit diffusion equation is described by a mixed fractional operator as in \eqref{main}. See Appendix B of \cite{dipierro2021nonlocal} for a thorough probabilistic discussion of this phenomenon. Such a phenomenon can also be seen in biological models, as explained in \cite{MPV13MixedLocalNonlocalBiology}, \cite{PV18MixedLocalNonlocalBiology}, or \cite{DV21MixedLocalNonlocalBiology}, and these mixed operators describe biological species whose individuals diffuse by a mixture of random walks and jump processes, according to prescribed probabilities. Indeed, mixed operators allow us to study the inter-correlating impact of local and nonlocal effects, such as in \cite{Blazevski2013MixedLocalNonlocalGeomagnetism}, and is a powerful tool in applied sciences. In fact, a classical model involving mixed fractional orders is that of the surface quasi-geostrophic equation, which is a semilinear anisotropic advection-diffusion equation involving a fractional operator and the classical gradient (see for instance \cite{WangWuLiChen2014DCDSCompressibleSQG} or \cite{ConstantinNguyen2018PhysDGlobalSQGBoundedDomains}).

We remark that a conceptually different yet closely related operator is those of the ones of variable exponents, and previous works include \cite{Hoh2000VariableExponent}, \cite{BassKassmann2005VariableExponent}, \cite{BassKassmann2005VariableExponentHolder}, \cite{ZengBaiRadulescu2023FracPLapVariableExp-DetermineCoefs-noUCP} and \cite{Ok2023CV_HolderNonLocalVariable}. Such operators and their associated models arise in some physical phenomena (see for instance \cite{farquhar2018computational} and \cite{Weiss2019fractionalgeophysical}). Correspondingly, their inverse problems have been considered in \cite{KianSoccorsiYamamoto2018TimeFracVariableExponent} and \cite{InversePbTimeFracLapVariableExponent} for the time-fractional case, but there has not yet been any result for the space-fractional Laplacian. 

Despite the increased versatility for models based on fractional operators with mixed orders, the associated problems in consideration are very challenging due to the presence of both nonlocal and local effects. Indeed, it was shown in \cite{CabreSerra2016ExtensionProblem} that the associated  Caffarelli-Silvestre extension is way more complicated. As such, previous results have only been for the poly-fractional time operator, such as in \cite{LinNakamura2019MultiTermTimeFracUCP,LinNakamura2021MultiTermTimeFracUCP}. To the best of our knowledge, there have not yet been any results on unique continuation properties of poly-fractional operators in the space domain, let alone any results on inverse problems involving such operators.

In this paper, we consider these poly-fractional operators $ P( (-\Delta_g)^s)$. We first give some novel unique continuation properties associated to them. Such properties are essential for the study of their associated exterior value problems \eqref{main}. Then as in \cite{GSU2020CalderonSchrodinger} and later works, we derive the uniqueness of various environmental effects in the model from the knowledge of the Dirichlet-to-Neumann map 
\[
    f\mapsto \left. P( (-\Delta_g)^s) u_f \right|_{\Omega}.
\]
We recover the potential, the source function in the semilinear case, and the coefficients associated to the non-isotropy of the fractional operator. Therefore, our problems can be viewed as variants of the fractional Calder\'on problem studied in \cite{GSU2020CalderonSchrodinger}. 
We believe that our work opens up the doors for poly-fractional inverse problems, and holds promise for addressing a multitude of real-world problems. 

It should be noted that another form of poly-fractional inverse problems has previously been considered, where the fractional problem takes the form of a Caputo fractional-time derivative and a space-fractional Laplacian of Mittag-Leffler type. For more details of the problem setup, refer to \cite{TU2013SimultaneousFracExpTimeSpace}, \cite{TTU2016SimultaneousFracExpTimeSpace}, \cite{Guerngar2020UniquenessFracExponent}, \cite{Guerngar2021SimultaneousFracExponent} or the references therein. In these works, the authors made use of the spectral eigenfunction expansion of the weak solution to the initial/boundary value problem, to recover the fractional exponents from the initial value function. Such a method heavily relies on the structural definition of the fractional Laplacian, which is difficult to be generalised.

\subsection{Organisation of This Paper}
The rest of this paper is structured as follows. In Section \ref{sec:Prelims}, we provide rigorous mathematical formulations of the fractional operator $P((-\Delta_g)^s)$ for $s=(s_1,\dots,s_M)$, $0<s_1<\cdots,s_M<\infty$ and fractional Sobolev spaces. In Section \ref{sec:UCP}, we prove the unique continuation properties of $\tilde{P}((-\Delta_{\tilde{g}})^{\tilde{s}})$, which will form the basis of the study of our inverse problems. In Section \ref{sec:IP}, we prove the uniqueness in recovering the potential, the source function in the semilinear case, and the coefficients associated to the non-isotropy of the fractional operator. We end with some final remarks and open problems in Section \ref{sec:final}.

\section{Preliminaries and Statement of Main Results}\label{sec:Prelims}
	
\subsection{Fractional Sobolev spaces}
	
We first begin with a review of fractional Sobolev spaces. 

The fractional Sobolev spaces $H^s(\mathbb{R}^n)$ for all real positive $s$ are defined by
\begin{equation}\label{Hsnorm}H^s(\mathbb{R}^n)=\{u\in \mathcal{S}'(\mathbb{R}^n):\{\xi\mapsto(1+|\xi|^2)^{s/2}\hat{u}(\xi)\}\in L^2(\mathbb{R}^n)\},\end{equation} with norm \[\norm{u}_{H^s(\mathbb{R}^n)}=\norm{(1+|\xi|^2)^{s/2}\hat{u}}_{L^2(\mathbb{R}^n)},\] and its dual space \begin{equation}\label{Hdualtildedef}H^{-s}(\mathbb{R}^d):=\{\xi\in \mathcal{S}'(\mathbb{R}^d):\{1+|\xi|^{-s}\hat{\xi}\}\in L^2(\mathbb{R}^d)\},\end{equation} where $\mathcal{S}$ is the Schwartz space and $\mathcal{S}'$ the dual, and $\hat{u}(\xi)=\int_{\mathbb{R}^d}e^{-2\pi ix\cdot\xi}u(x)\,dx$ is the Fourier transform of $u$. 

Let $\omega\subset\mathbb{R}^n$ be an open set. We define the following fractional Sobolev spaces
\[\tilde{H}^s(\omega) := \text{ closure of }C_c^\infty(\omega) \text{ in } H^s(\mathbb{R}^n),\]
\[H^s(\omega) := \{u|_\omega : u\in H^s(\mathbb{R}^n)\},\]
which is complete with the norm 
\[\|u\|_{H^s(\omega)}:=\inf\left\{ \|v\|_{H^s(\mathbb{R}^n)}:\,v\in H^s(\mathbb{R}^n)\text{ and }v|_\omega=u\right\}.\]
We also define
\[H^s_0(\omega) := \text{ closure of }C_c^\infty(\omega) \text{ in } H^s(\omega)\]
and 
\[H_{\overline{\omega}}^{s}(\mathbb{R}^n):=\left\{u\in H^s(\mathbb{R}^{n}):supp(u)\subset\overline{\omega}\right\}.\]
Observe that $\tilde{H}^s(\omega)\subset H^s_0(\omega)$ and the duals $(\tilde{H}^s(\omega))^*= H^{-s}(\Omega)$ and $(H^s(\omega))^*=\tilde{H}^{-s}(\omega)$. If $\omega$ is, in addition, a Lipschitz domain, $\tilde{H}^s(\omega)=H_{\overline{\omega}}^{s}(\mathbb{R}^n)$ for all $s\in\mathbb{R}$ and $H^s_0(\omega)=H_{\overline{\omega}}^{s}(\mathbb{R}^n)$ for $r\geq0$, $r\neq \frac{2k+1}{2}$, $k\in\mathbb{N}$. For a more detailed discussion regarding these Sobolev spaces, we refer readers to the reference \cite{McLean2000book}.

\subsection{Nonlocal Operators Defined via the Fourier Transform}\label{subsec:OperatorDef}

We next give the definition of the anisotropic nonlocal operator $(-\Delta_\gamma)^\sigma$ via the spectral characterization, for $\sigma\in(0,\infty)$. We first recall that in the isotropic case, the higher order fractional Laplacian $(-\Delta)^\sigma$ is defined for smooth, compactly supported functions via the Fourier transform for all $0<\sigma<\infty$:
\begin{equation}\label{IsotropicFracLapDef}
    \widehat{(-\Delta)^\sigma u} = |\xi|^{2\sigma}\hat{u},\quad \hat{u}(\xi)=\int_{\mathbb{R}^n}e^{-i\xi\cdot x}u(x)\,dx.
\end{equation}

Next, suppose that $\gamma(x) \in C^\infty(\mathbb{R}^n)$ is a symmetric positive definite matrix-valued function such that there exists $\lambda\in (0, 1)$,
\[
\lambda|\xi|^2\leq \xi^T\gamma(x)\xi\leq \lambda^{-1}|\xi|^2\quad \text{ for all } x\in\Omega \text{ and } \xi\in\mathbb{R}^n.  
\]
Then, $-\Delta_\gamma:=-\nabla \cdot \gamma(x) \nabla$ is a general second order linear non-negative self-adjoint elliptic operator, which is densely defined on $L^{2}(\mathbb{R}^{n})$ for $n\geq1$ \cite{stinga2010extension}. By \cite{StingaFracLapTorusHighOrder}, one can define the higher order fractional Laplacian operator $(-\Delta_\gamma)^\sigma := (-\nabla \cdot \gamma\nabla)^\sigma$ for any $\sigma>0$, $s\not\in\mathbb{N}$ in a spectral way as
\begin{equation*}
(-\Delta_\gamma)^\sigma=(-\nabla \cdot \gamma\nabla)^\sigma:=\int_{0}^{\infty}\lambda^\sigma dE(\lambda),
\end{equation*} 
where $\left\{ E(\lambda)\right\} $ is the unique spectral resolution of $-\Delta_\gamma$ and $dE(\lambda)$ is a regular Borel complex measure of bounded variation. The domain of $(-\Delta_\gamma)^\sigma$ is inherited from this spectral representation, and it holds that $Dom(-\Delta_\gamma)^\sigma\supset Dom(-\Delta_\gamma)^{\lfloor\sigma\rfloor}$, where $\lfloor\sigma\rfloor$ is the integer part of $\sigma$.

In addition, we assume that $\gamma:\mathbb{R}^n\to\mathbb{R}^n\times\mathbb{R}^n$ is bounded and locally $C^{2\lfloor\sigma\rfloor,1}$ such that the homogeneous H\"older norm $[\cdot]_{\dot{C}^{k,\alpha}}$ is small, in the sense that
\[[\gamma^{ij}]_{\dot{C}^{2\lfloor\sigma\rfloor,1}(B_4)}+[\gamma^{ij}]_{\dot{C}^{0,1}(B_4)}\ll\epsilon\] for some small parameter $\epsilon>0$, where $B_4:=\{x\in\mathbb{R}^n:|x|\leq4\}$ is the unit ball. We also assume that $\gamma^{ij}(0)=\delta^{ij}$ for the Kronecker delta function $\delta$.
 
In this paper, we always assume that $\gamma$ satisfies these assumptions. With these definitions, we are ready to state the forward problem.

We first begin by recalling a Poincar\'e inequality and a Sobolev inequality for higher order fractional Laplacians.

\begin{lemma}[Poincar\'e inequality, given in Theorem 3.7 of \cite{CMR2021HigherOrderFracLapUCP}]\label{Poincare}
    Let $s\in\mathbb{R}_+\backslash\mathbb{Z}$, $K\subset\mathbb{R}^n$ be a compact set, and $u\in H^s_K(\mathbb{R}^n)$. Then there exists a constant $C>0$ depending on $n$, $K$ and $s$ such that 
    \[\norm{u}_{L^2(\mathbb{R}^n)} \leq C\norm{(-\Delta)^{s/2}u}_{L^2(\mathbb{R}^n)}.\]
\end{lemma}

\begin{lemma}[Sobolev inequality, given in Theorem 2.2 of \cite{railo2022fractional}]
    Let $\Omega\subset\mathbb{R}^n$ be a bounded open set and $u\in \tilde{H}^s(\Omega)$. Suppose that $0\leq r\leq s<\infty$. Then there exists a constant $C>0$ depending on $n$, $r$, $s$ and $\Omega$ such that 
    \[\norm{(-\Delta)^{r/2} u}_{L^2(\mathbb{R}^n)} \leq C \norm{(-\Delta)^{s/2} u}_{L^2(\mathbb{R}^n)}.\]
\end{lemma}
Note that the two constants $C$ in these two inequalities are different.

For any operator
\[A((-\Delta_g)^\alpha) := \alpha_1 (-\Delta_{g_1})^{a_1} + \alpha_2 (-\Delta_{g_2})^{a_2} + \cdots + \alpha_p (-\Delta_{g_p})^{a_p},\] 
with $\alpha=(\alpha_1,\dots,\alpha_p)$, $0\leq a_1<\cdots<a_p<\infty$, $g=(g_1,\dots,g_p)$ for some $p\in\mathbb{N}$, $p<\infty$, $a_i\in\mathbb{R}$ and $\alpha_i>0$, $\alpha_i\in L^\infty(\mathbb{R}^n)$ for $i=1,\dots,p$,
we define the operator 
\[A_q:=A((-\Delta_g)^\alpha) + q \quad\text{ in }\Omega\]
with 
\[A_q(u,v):=\left\langle \alpha_1 ((-\Delta_{g_1})^{a_1}u), v \right\rangle + \left\langle \alpha_2 ((-\Delta_{g_2})^{a_2}u), v \right\rangle + \cdots + \left\langle \alpha_p ((-\Delta_{g_p})^{a_p}u), v\right\rangle + (qu,v)_\Omega.\]
Here, we write
\begin{align*}
\langle \varphi,\psi \rangle :=&\int_{\mathbb{R}^n}\varphi \psi \, dx, \text{ for any }\varphi,\psi \in L^2(\mathbb{R}^n),\\
(\varphi,\psi)_{\Omega}:=&\int_{\Omega} \varphi \psi \, dx, \text{ for any }\varphi ,\psi \in L^2(\Omega). 
\end{align*}

\begin{theorem}\label{ExistThm}
    Let $\alpha=(\alpha_1,\dots,\alpha_p)$, $0\leq a_1<\cdots<a_p<\infty$ for some $p\in\mathbb{N}$, $p<\infty$, $a_i\in\mathbb{R}$ and $\alpha_i>0$, $\alpha_i\in L^\infty(\mathbb{R}^n)$ for $i=1,\dots,p$, and $\gamma$ satisfies the assumptions above. Assume in addition that $A$ is a coercive bounded operator with domain $H^{a_p}(\mathbb{R}^n)$. 
    Let $f\in H^{\alpha_p}(\mathbb{R}^n)$ and $F\in H^{-\alpha_p}(\Omega)$. Then the problem
    \begin{equation}
        \begin{cases}
            A((-\Delta_g)^\alpha)u=F &\text{ in }\Omega,\\
            u=f &\text{ in }\Omega^c
        \end{cases}
    \end{equation}
    has a unique weak solution $u\in H^{\alpha_p}(\mathbb{R}^n)$, i.e. $u$ satisfies
    \[\langle A((-\Delta_g)^\alpha)u,v\rangle = \langle F,v\rangle\quad\text{ for all }v\in \tilde{H}^{\alpha_p}(\Omega),\]

    Moreover, $u$ satisfies 
    \begin{equation}\label{SolnEst}
        \sum_{i=1}^p\norm{u}_{H^{\alpha_i}(\mathbb{R}^n)}\leq C \left( \norm{F}_{ (H^\alpha(\Omega))^* } + \norm{f}_{ (H^\alpha(\Omega^c))^* } \right) ,
    \end{equation} 
    for some constant $C>0$ independent of $u$, $f$ and $F$. Here 
    \[
    \norm{\cdot}_{ (H^\alpha(\Omega))^*} := \sum_{i=1}^p\norm{F}_{ (H^{\alpha_i}(\Omega))^* }.
    \]
\end{theorem}

We skip the proof of Theorem \ref{ExistThm} here, since the proof follows by a simple application of the Lax-Milgram theorem, since $A$ is assumed to be coercive and bounded in $H^{\alpha_p}(\mathbb{R}^n)$. We note that a sufficient condition for $A$ to be coercive and bounded is when the coefficients $\alpha_i$ are Sobolev multipliers (see \cite{CMRU2022HigherOrderFracPolyUniqueness} for more details).

\begin{remark}
    Special cases of this result can be found in Lemma 3.4 of \cite{CMRU2022HigherOrderFracPolyUniqueness} with a single higher order fractional Laplacian and local lower order terms, in Theorem 1.1 of \cite{BiagiDipierroValdinociVecchi2022CommPDEMixedLocalNonlocal} with a local classical Laplacian and lower order fractional terms, in Lemma 5.1 of \cite{CMR2021HigherOrderFracLapUCP} with a single higher order fractional Laplacian and a single rough potential, and in \cite{CabreSerra2016ExtensionProblem} for infinitely many fractional Laplacians of order $0<s<1$. See also \cite{BiagiMeglioliPunzo2023-MixedLocalNonlocalUniqueness} which explores the uniqueness of the solution in the case of a local classical Laplacian and a single lower order fractional Laplacian, in the case where a zeroth order term in $A$ violates the coercivity of $A$.
\end{remark}

With the well-posedness at hand, we are able to define the corresponding DtN map $\mathcal{M}_{q,\alpha}$ rigorously, by
\begin{equation}
    \left\langle \mathcal{M}_{q,\alpha} f, g \right\rangle:= A_q (u_f, g),
\end{equation}
where $u_f$ is the unique solution to \eqref{main}, and $g\in C^\infty_0(\mathbb{R}^n)$ could be arbitrary.

\subsection{Admissibility Conditions}
\begin{definition}
    For abelian linear operators $(-\Delta_\gamma):\mathcal{S}(\mathbb{R}^n)\to L^2(\mathbb{R}^n)$, as defined in the previous section, acting on the Schwartz space $\mathcal{S}(\mathbb{R}^n)$, we define $\mathcal{A}$ to be the admissible algebra of all real posynomials $A((-\Delta_g)^a)$ with positive $\alpha_i(x)$ coefficients by 
    \[A((-\Delta_g)^a) := \alpha_1(x) (-\Delta_{g_1})^{a_1} + \alpha_2(x) (-\Delta_{g_2})^{a_2} + \cdots + \alpha_p(x) (-\Delta_{g_2})^{a_p}, \]
    $0\leq a_1<\cdots<a_p<\infty$, $p<\infty$, $a_i\in\mathbb{R}$, $p\in\mathbb{N}$, 
    such that $A$ is coercive and bounded in $H^{\alpha_p}(\mathbb{R}^n)\supset H^{s_M}(\mathbb{R}^n)$, and there exists another real posynomial $B$ with (possibly negative) $L^\infty(\mathbb{R}^n)$ coefficients 
    such that their (usual, pointwise-defined) product $BA$ is of the form 
    \[B((-\Delta_{g'})^{a'})A((-\Delta_g)^a) = \sum_{j=1}^{J} c_j(x) (-\Delta_{\breve{g}_j})^{n_j} + c_r(x) (-\Delta_{\breve{g}_0})^r,
    \]
    for $0\leq n_1<\cdots<n_J<\infty$, $r>0$, $J<\infty$ with $n_j\in\mathbb{N}$, $r\not\in\mathbb{N}$, $c_j,c_r\neq0$ for all $j=1,\dots,J$, for some $\breve{g}_j$.

    By density, this definition can be extended to any operator $(-\Delta_\gamma):L^2(\mathbb{R}^n)\to L^2(\mathbb{R}^n)$, with $(-\Delta_\gamma)u=\infty$ if $u\not\in H^{n_J}(\mathbb{R}^n)\cap H^r(\mathbb{R}^n)$.
\end{definition}

It is obvious that for all real posynomials $p\in\mathbb{R}[z]$, $p=\lambda_1z^{l_1}+\lambda_2z^{l_2}+\cdots+\lambda_Lz^{l_L}$, given an operator $T$ in the space of all abelian linear operators from $\mathcal{S}(\mathbb{R}^n)$ to $L^2(\mathbb{R}^n)$, we can write $p(T)=\lambda_1T^{l_1}+\lambda_2T^{l_2}+\cdots+\lambda_LT^{l_L}$, and the map $p\mapsto p(T)$ is a unital homomorphism. Therefore, our definition above makes sense.

Furthermore, the operator $(-\Delta_\gamma)^s$ is clearly commutative for any $0<s<\infty$ by definition.

It should be remarked that the admissible set $\mathcal{A}$ is non-empty. We give here two examples.

\hypertarget{Eg1}{\emph{Example 1.}} Suppose that $g=(g_1,\dots,g_p)$ is such that $g_i=\gamma_i\hat{g}$ for some constant $\gamma_i$, and $\hat{g}$ satisfying the assumptions of the previous section. For $p\geq2$, if the constant coefficients $\alpha_i$ of $A((-\Delta_g)^a) := \sum_{i=1}^p \alpha_i(-\Delta_{g_i})^{a_i}$ satisfy the relation $2\alpha_i=\alpha_{i-1}+\alpha_{i+1}$, then it can be easily seen that $A\in\mathcal{A}$. On the other hand, for $p=2$, it can also be easily verified that if  the constant coefficients $\alpha_1,\alpha_2$ are such that $2(\alpha_2-\alpha_1)\not\in\mathbb{N}$, then $A\in\mathcal{A}$. Note that for any $p$, we can take $B((-\Delta_{g'})^{a'})$ such that $g'_i=\gamma'_i\hat{g}$ for some constant $\gamma_i'$. 

\hypertarget{Eg2}{\emph{Example 2.}} Suppose $A((-\Delta_g)^a)$ is of the form $\sum_{n=1}^N \alpha_n(x) (-\Delta_{g_n})^n + \alpha_0(x)(-\Delta_{g_0})^s$ for any $n$ which may be larger or smaller than $s$. Then it suffices to take $B(T)=1$ to satisfy the admissibility condition. Note that this covers the case of a single higher order fractional Laplacian with lower order local perturbation in \cite{CMRU2022HigherOrderFracPolyUniqueness}.

\subsection{Mathematical Setup and Statement of the Main Results}

With these definitions, we are able to state our problem concretely. 
	
Let $\Omega \subset \mathbb{R}^n$ be a bounded Lipschitz domain, for $n\in \mathbb{N}$, and $q\in L^\infty(\Omega)$. Given $s=(s_1,\dots,s_M)$, $0<s_1<\cdots<s_M<\infty$, $s_M\in\mathbb{R}_+\backslash\mathbb{Z}$, we consider the exterior value problem 
\begin{equation}\label{main}
    \begin{cases}
        P( (-\Delta_g)^s)u(x) + q(x)u(x) := \sum_{i=1}^M \alpha_i(x)(-\Delta_{g_i})^{s_i}u(x) + q(x)u(x) =0 &\text{ in }\Omega, \\
        u(x)=f(x) &\text{ in }\Omega^c:=\mathbb{R}^n\backslash \overline{\Omega}.
    \end{cases}
\end{equation}
As discussed in the Introduction, we assume that $P( (-\Delta_g)^s) + q$ is a coercive bounded operator in $H^{s_M}(\mathbb{R}^n)$ and satisfies \eqref{qCond}. 
    
We are interested in the unique determination inverse problem for fractional polyharmonic equations. Let $\alpha_i^j,q_j\in L^\infty(\mathbb{R}^n)$, consider 
\begin{equation}\label{mainJ}
    \begin{cases}
        P( (-\Delta_g)^s)_j u_j + q_ju_j := \sum_{i=1}^M \alpha_i^j(-\Delta_{g_i})^{s_i}u_j + q_ju_j =0 &\text{ in }\Omega, \\
        u_j=f_j &\text{ in }\Omega^c
    \end{cases}
\end{equation} 
We also define 
\begin{equation}\label{tildePDef}
    \tilde{P}( (-\Delta_{\tilde{g}})^{\tilde{s}})u := \sum_{i=1}^{M'} \tilde{\alpha}_i(-\Delta_{\tilde{g}_i})^{\tilde{s}_i}u \quad \text{ for }\tilde{s}=(\tilde{s}_1,\dots,\tilde{s}_{M'}), 0\leq\tilde{s}_1<\cdots<\tilde{s}_{M'}\leq s_M
\end{equation} 
for $\tilde{P}\in\mathcal{A}$, $\tilde{g}=(\tilde{g}_1,\dots,\tilde{g}_{M'})$. 
Then, given the well-posedness of \eqref{main}, we define the Dirichlet-to-Neumann (DtN) map  
\begin{equation}
    \mathcal{M}_{q_j,\alpha^j} : H^{s_M}(\Omega^c)\to H^{-s_M}(\Omega), \quad f\mapsto \left. \tilde{P}( (-\Delta_{\tilde{g}})^{\tilde{s}}) u_f \right|_{\Omega} ,\quad j=1,2,
\end{equation}
where $u_f \in H^{s_M}(\mathbb{R}^n)\subset Dom(P)\cap Dom(\tilde{P})$ is the unique solution to \eqref{main}. Here, $Dom(P)$ denotes the domain of the operator $P( (-\Delta_g)^s)$, and we write similarly for $Dom(\tilde{P})$. 

Then, we can prove the following result:
\begin{theorem}[First uniqueness result for potential $q$]\label{MainThm1q}
    Let $\Omega \subset \mathbb{R}^n$ be a bounded Lipschitz domain, for $n\in \mathbb{N}$, and $q_j\in L^\infty(\Omega)$ for $j=1,2$. Suppose that for given nonzero $f=f_1=f_2\in  H^{s_M}(\Omega^c)$, the non-trivial unique solutions $u_{f_1},u_{f_2}\in H^{s_M}(\mathbb{R}^n)\cap C_c^\infty(\mathbb{R}^n)$ is such that $u_{f_1}(x)\neq0$ in some subset $E\subset\Omega$. Then 
    \begin{equation}
		\mathcal{M}_{q_1,\alpha}f = \mathcal{M}_{q_2,\alpha}f
    \end{equation}
    implies $q_1=q_2$ in $E$.
\end{theorem}

We always assume that $g=(g_1,\dots,g_M)$ is fixed. 

Furthermore, we have the following result under additional assumptions.

\begin{theorem}[Second uniqueness result for potential $q$]\label{MainThm2q}
    Let $\Omega \subset \mathbb{R}^n$ be a bounded Lipschitz domain, for $n\in \mathbb{N}$, and $q_j\in L^\infty(\Omega)$ for $j=1,2$. Let $W\subset \Omega^c$ be a nonempty open set. Suppose that for given nonzero $f=f_1=f_2\in C^\infty_c (W)$, the non-trivial unique solutions $u_{f_1},u_{f_2}\in H^{s_M}(\mathbb{R}^n)\cap C_c^\infty(\mathbb{R}^n)$ is such that $u_{f_1}(x)\neq0$ in some subset $E\subset\Omega$, and satisfies $\mathcal{L}u(x)=g(x)$ in $\Omega^c$ for some (local) second order elliptic operator $\mathcal{L}$ and $g(x)\in C^\infty_c(\overline{\Omega^c})$. Then, the relation 
    \begin{equation}
		\mathcal{M}_{q_1,\alpha}f = \mathcal{M}_{q_2,\alpha}f
    \end{equation}
    implies $q_1=q_2$ in $E$.
\end{theorem} 

Furthermore, we can extend our results to consider semilinear equations of the form 
\begin{equation}\label{mainSemilin}
    \begin{cases}
        P( (-\Delta_g)^s)u + F(x,u) =0 &\text{ in }\Omega, \\
        u=f &\text{ in }\Omega^c.
    \end{cases}
\end{equation}
Suppose $F$ is analytic up to order $L$, i.e. $F$ can be written as a power series 
\[F(x,u)=\sum_{\ell=0}^L F^{(\ell)}(x)\frac{u^\ell}{\ell!},\]
where $ F^{(\ell)}(x)=\frac{\partial^\ell F}{\partial u^\ell}(x,0)\in L^\infty(\Omega)$. Then we have the following corresponding results:
\begin{theorem}[Uniqueness result for source $F$]\label{MainThm1Semilin}
    Let $\Omega \subset \mathbb{R}^n$ be a bounded Lipschitz domain, for $n\in \mathbb{N}$, and $F_j$ with the analyticity defined above for $j=1,2$. For any nonzero $f=f_1=f_2\in  H^{s_M}(\Omega^c)$ with corresponding solutions $u_{f_1},u_{f_2}\in H^{s_M}(\mathbb{R}^n)\cap C_c^\infty(\mathbb{R}^n)$ such that $u_{f_1}(x)\neq0$ in some subset $E\subset\Omega$, 
    \begin{equation}
		\mathcal{M}_{F_1,\alpha}f = \mathcal{M}_{F_2,\alpha}f, 
    \end{equation}
    implies $F_1=F_2$ in $E$, with $F^{(0)}_1\equiv F^{(0)}_2$ in $\Omega$.
    
    Suppose instead we take nonzero $f_1,f_2\in C^\infty_c (W)$ for a nonempty open set $W\subset \Omega^c$, such that $u_{f_1},u_{f_2}$ are as in Theorem \ref{MainThm2q}. Then the same result holds.
\end{theorem} 

\begin{remark}
    Note that $F^{(0)}$ can be viewed as the constant term in \eqref{mainSemilin}.
\end{remark}

Similar results can also be said for the coefficients $\alpha_i$.

\begin{theorem}[Uniqueness result for coefficients $\alpha_i$]\label{MainThm1Coef}
    Let $\Omega \subset \mathbb{R}^n$ be a bounded Lipschitz domain, for $n\in \mathbb{N}$, and $q\in L^\infty(\Omega)$ for $j=1,2$. Assume that $\alpha_i^1=\alpha_i^2$ for every $i=1,\dots,M$ except $i=m$. Suppose that a given nonzero $f=f_1=f_2\in  H^{s_M}(\Omega^c)$, the non-trivial unique solutions $u_{f_1},u_{f_2}\in H^{s_M}(\mathbb{R}^n)\cap C_c^\infty(\mathbb{R}^n)$ is such that $(-\Delta_{g_m})^{s_m}u_{f_1}(x)\neq0$ in some subset $E'\subset\Omega$. Then 
    \begin{equation}\label{MainThm1CoefEq}
		\mathcal{M}_{q,\alpha^1}f = \mathcal{M}_{q,\alpha^2}f
    \end{equation}
    implies $\alpha_m^1=\alpha_m^2$ in $E'$.

    Suppose instead that for a given nonzero $f=f_1=f_2\in C^\infty_c (W)$, $W\subset \Omega^c$ is a nonempty open set, the non-trivial unique solutions $u_{f_1},u_{f_2}\in H^{s_M}(\mathbb{R}^n)\cap C_c^\infty(\mathbb{R}^n)$ is such that $(-\Delta_{g_m})^{s_m}u_{f_1}(x)\neq0$ in some subset $E'\subset\Omega$, and satisfies $\mathcal{L}u(x)=g(x)$ in $\Omega^c$ for some (local) second order elliptic operator $\mathcal{L}$ and $g(x)\in C^\infty_c(\overline{\Omega^c})$. Then, the relation 
    \begin{equation}
		\mathcal{M}_{q,\alpha^1}f = \mathcal{M}_{q,\alpha^2}f
    \end{equation}
    still implies $\alpha_m^1=\alpha_m^2$ in $E'$.
\end{theorem}

Observe that we show the results for the potential and coefficient using a single measurement, but we used an infinite number of measurements for the semilinear case. This is because we will be using the higher order linearisation method for the semilinear result, which requires an infinite number of measurements. 

Moreover, our proof relies on the following UCPs.

\begin{theorem}[UCP1]\label{UCPThm2}
    Let $u\in H^{s_M}(\mathbb{R}^n)\cap C_c^\infty(\mathbb{R}^n)$. If $u$ satisfies
    \begin{equation}\label{UCPCond2}
		\begin{cases}
            \tilde{P}( (-\Delta_{\tilde{g}})^{\tilde{s}})u = 0 & \text{ in }\Omega,\\
            u=0 & \text{ in } \Omega^c,
		\end{cases}
    \end{equation}
    then $u\equiv 0$ in $\mathbb{R}^n$.
\end{theorem}

Note that this can be extended by density to $u\in H^{s_M}(\mathbb{R}^n)$ where \eqref{UCPCond1} is well-defined. Refer to Theorem 2.4 of \cite{HitchhikerGuide} for the density result. On the other hand, we are only considering $u$ such that $u\in H^{s_M}(\mathbb{R}^n)\cap C_c^\infty(\mathbb{R}^n)$.

The condition \eqref{UCPCond2} can be interpreted as follows: We can assume that the value function is $0$ away from the domain $\Omega$, which is usually the medium in which the particle/species lives. Alternatively, it is also possible to input a function $f$ which is defined everywhere in the exterior $\Omega^c$ (for instance, $f$ may be nonlocal but concentrated near the boundary $\partial\Omega$).

On the other hand, if we assume that the value function is the solution of a second order linear elliptic problem, such as in most diffusion processes in non-viscous fluids, we have the following corollary:
	
\begin{theorem}[UCP2]\label{UCPThm3}
    Let $W$ be a nonempty open subset in $\Omega^c$. Let $\mathcal{L}$ be any given second order elliptic operator. For any $u\in H^{s_M}(\mathbb{R}^n)\cap C_c^\infty(\mathbb{R}^n)$ satisfying
    \begin{equation}\label{UCPCond3}
		\begin{cases}
            \tilde{P}( (-\Delta_{\tilde{g}})^{\tilde{s}})u = 0 & \text{ in }\Omega,\\
            \mathcal{L}u=0 &\text{ in }\Omega^c,\\
            u=0 & \text{ in } W,
		\end{cases}
    \end{equation}
    we have that $u\equiv 0$ in $\mathbb{R}^n$.
\end{theorem}
 
\section{Unique Continuation Properties}\label{sec:UCP}
	
In this section, we prove more general unique continuation properties (UCPs) for fractional polyharmonic equations. 

We first begin by stating the UCP for a single higher order anisotropic fractional Laplacian. We give the result of \cite[Proposition 1.9]{GarciaRuland2019UCPHigherOrderFracLap}, in our specific case. 
\begin{theorem}\label{UCPSingleThm}
    Suppose $\sigma\in\mathbb{R}_+\backslash\mathbb{N}$ and $u\in Dom((-\Delta_\gamma)^\sigma)$ for $\gamma$ satisfying the assumptions of Section \ref{subsec:OperatorDef}. If
    \[(-\Delta_\gamma)^\sigma u=u|_\omega=0\]
    for some nonempty open set $\omega\subset\mathbb{R}^n$, then $u\equiv0$ in $\mathbb{R}^n$.
\end{theorem}
A strong UCP result was first stated for the isotropic case as Corollary 5.5 in \cite{Yang2013HigherOrderFracLap} and proved in \cite{FelliFerrero2020HigherOrderFracLapUCP} for $1<s<2$, where the vanishing of $u$ of infinite order at a point $x_0\in\omega\subseteq\mathbb{R}^n$ implies $u\equiv0$ in $\omega$ for a connected open domain $\omega$. This was later extended in \cite[Theorem 1.2]{CMR2021HigherOrderFracLapUCP} to fractional orders $s\in(-n/4,\infty)\backslash\mathbb{Z}$, and in \cite{GarciaRuland2019UCPHigherOrderFracLap} for fractional Schr\"odinger equations with Hardy type gradient potentials for $s\in\mathbb{R}_+\backslash\mathbb{N}$. For other similar works, see \cite{SeoHigherOrderUCP1,SeoHigherOrderUCP2,SeoHigherOrderUCP3} in the case where $s$ depends on the dimension $n$, \cite{CMRU2022HigherOrderFracPolyUniqueness} for the UCP result with lower order ($<s$) perturbation, and \cite{KarRailoZimmermann2023HigherOrderFracPLapUCP} for the UCP result for fractional $p$-biharmonic operator.

With Theorem \ref{UCPSingleThm} in hand, we proceed to prove our main UCP results, Theorems \ref{UCPThm2} and \ref{UCPThm3}.

We first begin with a lemma.

\begin{lemma}\label{UCPThm1}
    \sloppy Let $W$ be a nonempty open subset in $\Omega^c$. If $u\in H^{s_M}(\mathbb{R}^n)\cap C_c^\infty(\mathbb{R}^n)$ satisfies
    \begin{equation}\label{UCPCond1}
		\begin{cases}
            \tilde{P}( (-\Delta_{\tilde{g}})^{\tilde{s}})u = 0 & \text{ in }\mathbb{R}^n,\\
            u=0 & \text{ in } W,
		\end{cases}
    \end{equation}
    then $u\equiv 0$ in $\mathbb{R}^n$.
\end{lemma}

\begin{proof}
    Since $\tilde{P}\in\mathcal{A}$, there exists $B((-\Delta_{g'})^b)$ such that 
    \[B((-\Delta_{g'})^b)\tilde{P}( (-\Delta_{\tilde{g}})^{\tilde{s}}) u(x)= \sum_{j=1}^{J} c_j (-\Delta_{\breve{g}_j})^{n_j} u(x) + c_r (-\Delta_{\breve{g}_0})^r u(x),\quad J<pq,x\in\mathbb{R}^n,\]
    where $n_j\in\mathbb{N}$, $r\not\in\mathbb{N}$, $c_j,c_r\neq0$ for all $j=1,\dots,J$. 
    By assumption, $\tilde{P}( (-\Delta_{\tilde{g}})^{\tilde{s}})=0$ in $\mathbb{R}^n$, so 
    \begin{equation}\label{UCPThm1PfEq1}B((-\Delta_{g'})^b)\tilde{P}( (-\Delta_{\tilde{g}})^{\tilde{s}}) u(x)= \sum_{j=1}^{J} c_j (-\Delta_{\breve{g}_j})^{n_j} u(x) + c_r (-\Delta_{\breve{g}_0})^r u(x)=0\quad\text{ for all }x\in\mathbb{R}^n.\end{equation} 
    In particular, this holds for $x\in W$, where we have assumed $u(x) = 0$, which gives \[(-\Delta_{\breve{g}_j})^{n_j} u(x) = 0\quad\text{ in }W\text{ for every }n_j\in\mathbb{N}.\] Consequently, from \eqref{UCPThm1PfEq1}, we obtain \[(-\Delta_{\breve{g}_0})^r u(x) = 0\quad\text{ in }W.\]
    Applying the unique continuation principle for $(-\Delta_{\breve{g}_0})^r$ in $W$ given in Theorem \ref{UCPSingleThm}, we can obtain 
    \[u\equiv0\quad\text{ in }\mathbb{R}^n.\]
\end{proof}

\begin{remark}
    Consider the case of constant coefficients as in \hyperlink{Eg1}{\emph{Example 1}}. Assume additionally that the operators $(-\Delta_{\tilde{g}_i})^{\tilde{s}_i}$ are isotropic, in the sense that $\tilde{g}_i=\tilde{\gamma}_i\mathbb{I}$ for some positive constants $\tilde{\gamma}_i$ and $\mathbb{I}$ is the identity matrix, such that $(-\Delta_{\tilde{g}_i})^{\tilde{s}_i} = \underline{\gamma}_i(-\Delta)^{\tilde{s}_i}$ for corresponding constants $\underline{\gamma}_i$ and the fractional Laplacian $(-\Delta)^{\tilde{s}_i}$. In this case, we can conduct a Fourier transform on $(-\Delta)^{\tilde{s}_i}$ given by \eqref{IsotropicFracLapDef}. Then, this result can also be interpreted as follows:
    Suppose $u$ is smooth enough (say $L^1(\mathbb{R}^n)$) such that the Fourier transform of $\tilde{P}( (-\Delta_{\tilde{g}})^{\tilde{s}})u$ is well-defined. Then, formally,  
    \[\widehat{\tilde{P}( (-\Delta_{\tilde{g}})^{\tilde{s}})u}=\widehat{\tilde{P}}(|\xi|)\hat{u}(\xi),\] where $\widehat{\tilde{P}}(|\xi|)$ is of the form 
    \[\widehat{\tilde{P}}(|\xi|)=\sum_{i=1}^{M'}\tilde{\alpha}_i\tilde{\gamma}_i|\xi|^{2\tilde{s}_i}\] with $\tilde{\alpha}_i$ positive. Then, $\widehat{\tilde{P}}(|\xi|)=0$ on a set of measure $0$, so for any open set, $\hat{u}(\xi)=0$. Therefore, $\hat{u}(\xi)=0$ a.e. in $\mathbb{R}^n$, and its inverse Fourier transform $u(x)\equiv0$ in $\mathbb{R}^n$.

    Note that this does not hold in the general case for anisotropic fractional Laplacians $(-\Delta_{\tilde{g}})^{\tilde{s}}$, or even in the isotropic case when the coefficients $\tilde{\alpha}_i(x)$ are non-constant functions with respect to $x$, such as in \hyperlink{Eg2}{\emph{Example 2}}. In the first case, the Fourier transform of the anisotropic fractional Laplacians $(-\Delta_{\tilde{g}})^{\tilde{s}}$ does not have a well-defined form. In the second case, the non-constant coefficients $\tilde{\alpha}_i(x)$ will become a convolution with a power of $|\xi|$, and there may be open sets on which $\widehat{\tilde{P}}(|\xi|)=0$.
\end{remark}

From this, we are able to prove Theorem \ref{UCPThm2}.

\begin{proof}[Proof of Theorem \ref{UCPThm2}]
    Since $u\equiv0$ in the exterior of $\Omega$, we can view the condition \eqref{UCPCond2} as a Dirichlet problem for $u$ in $\Omega$ with Dirichlet condition $0$. Applying the existence Theorem \ref{ExistThm} with $F=f\equiv0$, we have that $u\equiv0$ by \eqref{SolnEst}. This means that $u$ satisfies 
    \begin{equation}
		\begin{cases}
            \tilde{P}( (-\Delta_{\tilde{g}})^{\tilde{s}})u = 0 & \text{ in }\mathbb{R}^n,\\
            u=0 & \text{ in } \Omega^c,
		\end{cases}
    \end{equation}
    which in particular, means that $u$ satisfies 
    \begin{equation}
		\begin{cases}
            \tilde{P}( (-\Delta_{\tilde{g}})^{\tilde{s}})u = 0 & \text{ in }\mathbb{R}^n,\\
            u=0 & \text{ in } W
		\end{cases}
    \end{equation}
    for any nonempty open subset $W\subset\Omega^c$, i.e. $u$ satisfies the condition \eqref{UCPCond1} of Lemma \ref{UCPThm1}. Therefore, we can apply the result of Lemma \ref{UCPThm1} to obtain that $u\equiv0$ in $\mathbb{R}^n$.    
\end{proof}

This result implies our final unique continuation principle.

\begin{proof}[Proof of Theorem \ref{UCPThm3}]
    Since $u$ satisfies 
    \begin{equation}
		\begin{cases}
            \mathcal{L}u=0 &\text{ in }\Omega^c,\\
            u=0 & \text{ in } W\subset\Omega^c,
		\end{cases}
    \end{equation}
    by maximum principles in the classical theory of second order elliptic operators, we have that $u=0$ in $\Omega^c$. By the previous Theorem \ref{UCPThm2}, we have the result $u\equiv0$ in $\mathbb{R}^n$.
\end{proof}

\section{Inverse Problems}\label{sec:IP}

\subsection{Recovery of Potential}
\begin{proof}[Proof of Theorem \ref{MainThm1q}]
    We first assume that $\alpha_i^1=\alpha_i^2$, i.e. $P( (-\Delta_g)^s)_1=P( (-\Delta_g)^s)_2=:P( (-\Delta_g)^s)$ in \eqref{mainJ}. Suppose $u_j$, $j=1,2$, satisfies \eqref{mainJ}, i.e. 
    \begin{equation}\label{mainJPfQ}
    \begin{cases}
        P( (-\Delta_g)^s) u_j + q_ju_j  =0 &\text{ in }\Omega, \\
        u_j=f_j &\text{ in }\Omega^c.
    \end{cases}
    \end{equation} 
    
    Since $ \mathcal{M}_{q_1,\alpha}f =\mathcal{M}_{q_2,\alpha}f $, for a given $\tilde{P}( (-\Delta_{\tilde{g}})^{\tilde{s}})$, 
    \[\tilde{P}( (-\Delta_{\tilde{g}})^{\tilde{s}})u_1 = \mathcal{M}_{q_1,\alpha}f = \mathcal{M}_{q_1,\alpha}f = \tilde{P}( (-\Delta_{\tilde{g}})^{\tilde{s}})u_2 \quad \text{ in }\Omega.\] 
    At the same time, 
    \begin{equation}\label{mainJPfQEq0}\left. u_1\right|_{\Omega^c} = f_1 = f_2 = \left. u_2\right|_{\Omega^c}\quad \text{ in }\Omega^c.\end{equation}
    Since $\tilde{P}$ is linear (any fractional power of the Laplacian is linear by the linearity of the Fourier transform), writing $\tilde{u}=u_1-u_2\in H^{s_M}(\mathbb{R}^n)$, we have that
    \begin{equation}
		\begin{cases}
            \tilde{P}( (-\Delta_{\tilde{g}})^{\tilde{s}})\tilde{u} = 0 & \text{ in }\Omega,\\
            \tilde{u}=0 & \text{ in } \Omega^c,
		\end{cases}
    \end{equation}
    i.e. $\tilde{u}$ satisfies \eqref{UCPCond2}. By the UCP Theorem \ref{UCPThm2}, we obtain that $\tilde{u}\equiv0$ in $\mathbb{R}^n$, i.e. $u_1=u_2$ in $\mathbb{R}^n$.

    In particular, $\tilde{u}=u_1-u_2=0$ in $\Omega$. Therefore, taking the difference between the two equations of \eqref{mainJPfQ}, we have 
    \begin{equation}\label{mainJPfQEq1}
    0 = P( (-\Delta_g)^s) \tilde{u} + (q_1-q_2)u_1 + q_2 \tilde{u} = (q_1-q_2)u_1 \quad \text{ in }\Omega.
    \end{equation} 
    But for a nonzero input $f_1$ in $\Omega^c$, the Lax-Milgram theorem guarantees a non-trivial solution $u_{f_1}$. Assume that $u_1(x)\neq0$ for all $x\in E\subset\Omega$. Restricting to $E$, we have that 
    \begin{equation}\label{mainJPfQEq2}(q_1(x)-q_2(x))u_1(x)=0\quad\text{ in }E,
    \end{equation}
    so $q_1(x)=q_2(x)$ for $x\in E$. 
\end{proof}

Observe that \eqref{mainJPfQEq1}--\eqref{mainJPfQEq2} are pointwise in $x\in E$. In fact, we are only using a single measurement here, unlike the infinite measurements usually required for previous results as in \cite{CMR2021HigherOrderFracLapUCP}, \cite{CMRU2022HigherOrderFracPolyUniqueness} or \cite{KarRailoZimmermann2023HigherOrderFracPLapUCP}.

\begin{proof}[Proof of Theorem \ref{MainThm2q}]
    The proof follows along the same lines as that of Theorem \ref{MainThm1q}. However, instead of \eqref{mainJPfQEq0}, we have that \begin{equation}\left. u_1\right|_W = f_1 = f_2 = \left. u_2\right|_W\quad \text{ in }W.\end{equation}
    Since $u_j$ is assumed to satisfy $\mathcal{L}u_j=g$ in $\Omega^c$ for some second order elliptic operator $\mathcal{L}$, consequently, $\tilde{u}=u_1-u_2\in H^{s_M}(\mathbb{R}^n)$ satisfies
    \begin{equation}
		\begin{cases}
            \tilde{P}( (-\Delta_{\tilde{g}})^{\tilde{s}})\tilde{u} = 0 & \text{ in }\Omega,\\
            \mathcal{L}\tilde{u}=0 &\text{ in }\Omega^c,\\
            \tilde{u}=0 & \text{ in } Q,
		\end{cases}
    \end{equation}
    i.e. $\tilde{u}$ satisfies \eqref{UCPCond3}. By the UCP Theorem \ref{UCPThm3}, we obtain that $\tilde{u}\equiv0$ in $\mathbb{R}^n$, i.e. $u_1=u_2$ in $\mathbb{R}^n$. Therefore, the remaining continues as in the previous proof, and we obtain $q_1(x)=q_2(x)$ for $x\in E$. 
\end{proof}

\subsection{Semilinear Case}

As in \cite{LinLiuLiuZhang2021-InversePbSemilinearParabolic-CGOSolnsSuccessiveLinearisation}, we make use of a higher order linearisation scheme, which we briefly sketch here. 
Consider the system \eqref{mainSemilin}. Let 
\[f(\varepsilon)=\sum_{\ell=0}^L\varepsilon^\ell f^\ell.\] 
By Theorem \ref{ExistThm}, there exists a unique solution $u(x;\varepsilon)$ of \eqref{mainSemilin}. Let $u(x;0)$ be the solution of \eqref{mainSemilin} when $\varepsilon=0$.

Define \[u^{(1)}:=\partial_\varepsilon u|_{\varepsilon=0}=\lim\limits_{\varepsilon\to 0}\frac{u(x,t;\varepsilon)-u(x,t;0) }{\varepsilon},\] and consider the corresponding problem for  $u^{(1)}$. Since $F$ is analytic, we have 
\begin{equation}\label{mainSemilin1Ord}
    \begin{cases}
        P( (-\Delta_g)^s)u^{(1)}(x) + F^{(1)}(x)u^{(1)}(x)=0 &\text{ in }\Omega, \\
        u^{(1)}=f^1 &\text{ in }\Omega^c.
    \end{cases}
\end{equation}

Next, we consider 
\[u^{(2)}:=\partial_\varepsilon^2 u|_{\varepsilon=0},\] which gives the second order linearisation:
\begin{equation}\label{mainSemilin2Ord}
    \begin{cases}
        P( (-\Delta_g)^s)u^{(2)}(x) + F^{(1)}(x)u^{(2)}(x) + F^{(2)}(x)[u^{(1)}(x)]^2=0 &\text{ in }\Omega, \\
        u^{(2)}=f^2 &\text{ in }\Omega^c.
    \end{cases}
\end{equation}

Inductively, for $\ell\in\mathbb{N}$, we consider 
\[u^{(\ell)}=\partial_\varepsilon^\ell u|_{\varepsilon=0},\]
we can obtain a sequence of equations, which shall be employed again in determining the higher order Taylor coefficients of the unknown $F$. 

\begin{remark}
    Note that in order to apply this higher order linearisation technique, we need the infinite differentiability of the equation \eqref{mainSemilin} with respect to the given boundary data $f$, which can be easily shown by applying the implicit function theorem of Banach spaces, as in \cite{LinLiuLiuZhang2021-InversePbSemilinearParabolic-CGOSolnsSuccessiveLinearisation}. We omit the proof here. 
\end{remark}

With this, we can proceed to prove Theorem \ref{MainThm1Semilin}.  
\begin{proof}[Proof of Theorem \ref{MainThm1Semilin}]
    Comparing \eqref{mainSemilin1Ord} with \eqref{main}, we apply the results of Theorem \ref{MainThm1q} to \eqref{mainSemilin1Ord} to obtain that $F^{(1)}_1(x)=F^{(1)}_2(x)$ for all $x\in E$ for some small enough $E\subset\Omega$, so we have the uniqueness result for the first order Taylor coefficient of $F$.

    Furthermore, $u^{(1)}_1(x)=u^{(1)}_2(x)$ in $\mathbb{R}^n$. Thus, \eqref{mainSemilin2Ord} reduces to 
    \begin{equation}
    \begin{cases}
        P( (-\Delta_g)^s)u^{(2)}(x) + F^{(2)}(x)u^{(2)}(x)=0 &\text{ in }\Omega, \\
        u^{(2)}=f^2 &\text{ in }\Omega^c:=\mathbb{R}^n\setminus \overline{\Omega}.
    \end{cases}
    \end{equation}
    Once again, we can compare this with \eqref{main} and apply Theorem \ref{MainThm1q} to obtain that $F^{(2)}_1(x)=F^{(2)}_2(x)$ for all $x\in E$, since $u^{(2)}(x)\neq0$ for all $x\in E\subset\Omega$.

    Reiterating this argument inductively for $\ell=1,\dots,L$, we have the uniqueness result for all the $\ell$-th order Taylor coefficients of $F$. Substituting this result in \eqref{mainSemilin} for $j=1,2$, we have that 
    \begin{equation*}
    \begin{cases}
        P( (-\Delta_g)^s)u_j + \sum_{\ell=1}^L F^{(\ell)}(x)\frac{u_j^\ell}{\ell!} + F^{(0)}_j =0 &\text{ in }\Omega, \\
        u_j=f_j &\text{ in }\Omega^c.
    \end{cases}
    \end{equation*}
    Taking the difference between the two equations for $j=1,2$, we have that, for $\mathcal{M}_{F_1,\alpha} = \mathcal{M}_{F_2,\alpha}$,
    \begin{equation*}
    \begin{cases}
        P( (-\Delta_g)^s)\tilde{u} + \sum_{\ell=1}^L F^{(\ell)}(x)\frac{\tilde{u}^\ell}{\ell!} + F^{(0)}_1 - F^{(0)}_2 =0 &\text{ in }\Omega, \\
        \tilde{u}=0 &\text{ in }\Omega^c,
    \end{cases}
    \end{equation*}
    where $\tilde{u}$ represents $u_1-u_2$ here. But by the UCP Theorem \ref{UCPThm2}, $\tilde{u}\equiv0$ in $\mathbb{R}^n$. Therefore, 
    \[F^{(0)}_1 - F^{(0)}_2 =0 \text{ in }\Omega.\]     
    By the definition of $F$, this means that $F_1=F_2$ for $x\in E$.

    Finally, the proof of the second part follows by a similar modification as in that of Theorem \ref{MainThm2q}. Indeed, for each order of linearisation $\ell=1,\dots,L$ (inductively), we have 
    \begin{equation}
    \begin{cases}
        P( (-\Delta_g)^s)(u^{(\ell)}_1-u^{(\ell)}_2) + (F^{(\ell)}_1-F^{(\ell)}_2)u^{(\ell)}_1+F^{(1)}_2(u^{(\ell)}_1-u^{(\ell)}_2)=0 &\text{ in }\Omega, \\
        \mathcal{L}(u^{(\ell)}_1-u^{(\ell)}_2)=0 &\text{ in }\Omega^c,\\
        u^{(\ell)}_1-u^{(\ell)}_2=0 &\text{ in }\Omega^c.
    \end{cases}
    \end{equation}
    Therefore, as in Theorem \ref{MainThm2q}, we can apply the UCP Theorem \ref{UCPThm3} to obtain the $F^{(\ell)}_1=F^{(\ell)}_2$ in $E$ for every $\ell=1,\dots,L$. The case $\ell=0$ follows similarly as in the first part.
\end{proof}

In this case, in order to conduct the linearisation for every $\varepsilon$, we require an infinite number of measurements to recover $F$. We are unable to use the result of \cite{LinLiu2022FracLapInverseProbMinMeasurements} to obtain the same result using a minimal number of measurements, because we have set up our UCP results differently.

\subsection{Recovery of Non-Isotropy}

Next, we proceed to recover the non-isotropy of the poly-fractional equation, given by the coefficients $\alpha_i$ of \eqref{main}.

\begin{proof}[Proof of Theorem \ref{MainThm1Coef}]

We first assume that for a nonzero input $f_j$ in $\Omega^c$, the non-trivial unique solution $u_j$ is such that for some $E'\subset\Omega$, $(-\Delta_{g_m})^{s_m} u_1(x)\neq0$ for all $x\in E'$.

Let $u_j$ be the solution of \eqref{mainJ} for $j=1,2$, that is, $u_j$ satisfies 
\begin{equation}\label{mainJnoQ}
    \begin{cases}
        P( (-\Delta_g)^s)_j u_j + qu_j := \sum_{i=1}^M \alpha_i^j(-\Delta_{g_i})^{s_i}u_j + qu_j =0 &\text{ in }\Omega, \\
        u_j=f_j &\text{ in }\Omega^c
    \end{cases}
\end{equation} 
Writing $\tilde{u}=u_1-u_2$, $\tilde{u}$ solves
\begin{equation}\label{mainJnoQDiff}
    \begin{cases}
        P( (-\Delta_g)^s)_1 \tilde{u} + q\tilde{u} = \left(\alpha_m^2-\alpha_m^1\right)(-\Delta_{g_m})^{s_m}u_2 &\text{ in }\Omega, \\
        \tilde{u}=0 &\text{ in }\Omega^c
    \end{cases}
\end{equation} 

Next, as in the proof of Theorem \ref{MainThm1q}, the condition \eqref{MainThm1CoefEq} 
and the UCP Theorem \ref{UCPThm2} implies $\tilde{u}=0$ in $\mathbb{R}^n$, so $P( (-\Delta_g)^s)_j\tilde{u}=q\tilde{u}=0$ in $\Omega$. Therefore,
\[\left(\alpha_m^2-\alpha_m^1\right)(-\Delta_{g_m})^{s_m}u_2 = 0 \quad \text{ in }\Omega.\] 
In particular, \[\left(\alpha_m^2(x)-\alpha_m^1(x)\right)(-\Delta_{g_m})^{s_m}u_2(x) = 0 \quad \forall x\in E'.\] 

But by assumption, $E'$ is taken to be the set in $\Omega$ such that $(-\Delta_{g_m})^{s_m} u_1(x)\neq0$. Therefore, 
\[\alpha_m^2(x)=\alpha_m^1(x) \quad \forall x\in E'.\]

Finally, the proof for the second case follows by a similar modification as in that of Theorem \ref{MainThm2q}. Here, instead of \eqref{mainJnoQDiff}, we have 
    \begin{equation}
    \begin{cases}
        P( (-\Delta_g)^s)_1 \tilde{u} + q\tilde{u} = \left(\alpha_m^2-\alpha_m^1\right)(-\Delta_{g_m})^{s_m}u_2 &\text{ in }\Omega, \\
        \mathcal{L}\tilde{u}=0 &\text{ in }\Omega^c,\\
        \tilde{u}=0 &\text{ in }\Omega^c.
    \end{cases}
    \end{equation}
    Apply the UCP Theorem \ref{UCPThm3} to obtain the $\alpha_m^2(x)=\alpha_m^1(x)$ in $E'$.
\end{proof}

Once again, as in the results for recovering the potential in Theorems \ref{MainThm1q}--\ref{MainThm2q}, we are only using a single measurement here.

\section{Final Remarks and Open Problems}\label{sec:final}
We remark that it may also be possible to apply our UCP result to recover the variable coefficient $\gamma(x)$ of the anisotropic fractional Laplacian $(-\Delta_\gamma)^\sigma$ or $(-\Delta)^\sigma_\gamma:=-\nabla^s\cdot \gamma\nabla^s$. A possibility is using the fractional Liouville reduction for the latter case. Similar results have been obtained in \cite{Covi2022_AnistropicFracConductivity-arXiv} in the case of $0<s<1$ for a single anisotropic fractional Laplacian $(-\Delta)^\sigma_\gamma$.

On another note, the spectral decomposition of the higher order fractional Laplacian $(-\Delta_\gamma)^\sigma$ may possibly be useful in the recovery of the exponent of the poly-fractional Laplacian. Indeed, previous works on recovering the exponent have made use of the eigenfunction expansion of the solution of the forward problem (see, for instance, \cite{Guerngar2020UniquenessFracExponent,Guerngar2021SimultaneousFracExponent}). 

Another interesting open problem is to consider the variable exponent fractional operator instead. This may include the form considered in \cite{ZengBaiRadulescu2023FracPLapVariableExp-DetermineCoefs-noUCP} and \cite{Ok2023CV_HolderNonLocalVariable}. A result is known in \cite{KianSoccorsiYamamoto2018TimeFracVariableExponent} and \cite{InversePbTimeFracLapVariableExponent} for the time-fractional case, but there has not yet been any result for the space-fractional Laplacian. Indeed, suppose the variable exponent $\sigma(x)$ of the fractional Laplacian is such that $\sigma(x)=1$ in $\Omega^c$. Then, the UCP condition \eqref{UCPCond1} reduces to that of \eqref{UCPCond3}.

Therefore, we leave these two problems as interesting open problems for readers for future research.

		\medskip 
	
	\noindent\textbf{Acknowledgment.} 

    The authors would like to thank Yi-Hsuan Lin and Philipp Zimmermann for helpful discussion. 
    C.-L. Lin is partially supported by the Ministry of Science and Technology of Taiwan.     
	H. Liu is supported by the Hong Kong RGC General Research Funds (projects 12302919, 12301420 and 11300821), the NSFC/RGC Joint Research Fund (project N CityU 101/21), and the France-Hong Kong ANR/RGC Joint Research Grant (A-CityU 203/19).

\bibliographystyle{plain}
\bibliography{ref}

\end{document}